\documentclass[12pt]{article}
\usepackage{}
\usepackage{amssymb}
\usepackage{amsmath} 

\usepackage{amsfonts}
\usepackage{appendix}
\usepackage{listings} 
       \font\tenmsb=msbm10
       \font\sevenmsb=msbm7
       \font\fivemsb=msbm5
       \catcode`\@=11
       \ifx\amstexloaded@\relax\catcode`\@=\active
       \endinput\else\let\amstexloaded@\relax\fi
       \def\spaces@{\space\space\space\space\space}
       \def\spaces@@{\spaces@\spaces@\spaces@\spaces@\spaces@}
       \def\space@.{\futurelet\space@\relax}
       \space@. %
       \def\Err@#1{\errhelp\defaulthelp@\errmessage{AmS-TeX error: #1}}
       \def\relaxnext@{\let\next\relax}
       \def\accentfam@{7}
       \def\noaccents@{\def\accentfam@{0}}
       \def\Cal{\relaxnext@\ifmmode\let\next\Cal@\else
       \def\next{\Err@{Use \string\Cal\space only in math mode}}\fi\next}
       \def\Cal@#1{{\Cal@@{#1}}}
       \def\Cal@@#1{\noaccents@\fam\tw@#1}
       \def\Bbb{\relaxnext@\ifmmode\let\next\Bbb@\else
       \def\next{\Err@{Use \string\Bbb\space only in math mode}}\fi\next}
       \def\Bbb@#1{{\Bbb@@{#1}}}
       \def\Bbb@@#1{\noaccents@\fam\msbfam#1}
       \newfam\msbfam
       \textfont\msbfam=\tenmsb
       \scriptfont\msbfam=\sevenmsb
       \scriptscriptfont\msbfam=\fivemsb

\def\skipaline{\removelastskip\vskip12pt plus 1pt minus 1pt}

\def\Proof{\removelastskip\skipaline
\noindent \it Proof. \rm}

\def\beq{\begin{equation}}
\def\eeq{\end{equation}}

\newtheorem{Theorem}{Theorem}[section]
\newtheorem{Lemma}{Lemma}[section]

\newtheorem{Example}{Example}[section]

\@addtoreset{equation}{section}

\newcommand{\proof}{\par\medbreak\it Proof.\quad\rm}

\title{The asymptotic power series of three kinds of sequences\footnote{
A Project Supported by NSF of Sichuan Province (2023NSFSC0065)
and Research and innovation team of Neijiang Normal University (17TD04).}}

 \author{Yong-Guo Shi\footnote{The corresponding author.
E-mail addresses: scumat@163.com (Y.-G. Shi)}
\\ \\
{\small College of Mathematics and Information Science, Neijiang Normal University}
\\
{\small Neijiang, Sichuan 641112, P. R. China}
}

\date{}

\begin{document}

\maketitle

\begin{quote}
\small {\bf Abstract.} We present a general method to obtain asymptotic power series for three kinds of sequences. And we give  recurrence relations for determining the coefficients of asymptotic power series for these sequences. As applications, we show how these theoretical results can be used to deduce approximation formulas for some well-known sequences and some integrals with a parameter.

\end{quote}

\small
{\it AMS(2020) subject classifications.  41A60, 45M05}


\hskip\parindent

\section{Introduction}

In asymptotic analysis, deriving an asymptotic expansion is generally considered to be technically difficult (cf. \cite{NMMortici2016,QMortici2015,XC1991}). 
The asymptotic expansion with respect to $n^{-1}$
$$
x_n =b_0+\frac{b_1}{n}+\frac{b_2}{n^2}+\cdots \quad \text { as } n \rightarrow \infty
$$
is called asymptotic power series \cite{Copson1965,Erdélyi1956}. 
When a sequence $x_n$ is approximated by the leading-order asymptotic estimate $y_n$,
there is a desire to require more precise approximations.
The first asymptotic expansion is the Stirling asymptotic expansion for the factorial function, which is known for centuries \cite{AbSt1970,Burić2019,Hsu1997}. Stirling’s series for the gamma function is due to Laplace
$$
\Gamma(n+1) \sim \sqrt{2 \pi n}\left(\frac{n}{e}\right)^n\left(1+\frac{1}{12 n}+\frac{1}{288 n^2}-\frac{139}{51840 n^3}-\frac{571}{2488320 n^4}+\ldots\right)
$$
as $n \rightarrow \infty$, see \cite[p. 257, Eq. (6.1.37)]{AbSt1970} and \cite{Copson1965,SriChZe2012}. 
Laforgia and Natalini \cite{LafNat2012} used Mortici's lemmas to study the asymptotic expansion of the ratio of two gamma functions.
Chen \cite{Chen2014} developed Windschitl’s approximation formula for the gamma function to two asymptotic expansions.

Many authors also investigated asymptotic expansions for Landau's constants  \cite{Chen2013,Chen2021CAM,Mortici2011}, Lebesgue's constants \cite{ChenCh2014,Nemes2012,Watson1930,Zhao2009} and Euler’s constant \cite{BuriElez2013,Chen2021CAM,Gourdon2004,Karatsuba2000,Mačys2013}.

In 2010, for one special kind of sequence, Mortici \cite{Mortici2010} presented
the following asymptotic expansion form 
$$
x_n = y_n \exp\left(\sum_{k=0}^{\infty} \frac{b_k}{n^k}\right).
$$
In this paper,
on the basis of his method, we
consider the following asymptotic expansion form 
$$
x_n = y_n \left(\sum_{k=0}^{\infty} \frac{b_k}{n^k}\right).
$$
And we present a general method to obtain asymptotic power series for three kinds of sequences.
Moreover, we abstractly refine Chen's method in \cite{Chen2021CAM} to the general form.
As applications, we present how these theoretical results can be used to deduce approximation formulas for some well-known sequences and some integrals with a parameter.

\section{Preliminaries}

For convenience, let $(s)_n$ be Pochhammer's symbol defined as
$$
(s)_0:=1,\quad (s)_n:=s(s-1)(s-2)\cdots(s-n+1), \,\,\text{for}\,\, n\geq 1.
$$
For two sequences $\{x_n\}$ and $\{y_n\}$, we write
$x_n=O(y_n)$ iff $|x_n / y_n|$ is bounded from above as $n \rightarrow \infty$, $x_n=o(y_n)$ iff $x_n / y_n \rightarrow 0$ as $n\rightarrow \infty$, and $x_n \sim y_n$ iff $x_n / y_n\rightarrow 1$ as $n \rightarrow \infty$.

\begin{Lemma}\label{ForwardShift}
Suppose that $f(n+1)$ has  asymptotic power series with respect to $(n+1)^{-k}$
$$
f(n+1) \sim \sum_{k=1}^{\infty} \frac{a_k}{(n+1)^k},\quad  x\to +\infty.
$$
Then $f(n+1)$ has a power expansion  with respect to $n^{-k}$ as follows
$$
f(n+1) \sim \sum_{k=1}^{\infty} \frac{a_k}{(n+1)^k}=\sum_{k=1}^{\infty} \sum_{j=1}^{k} a_j(-1)^{k-j}\left(\begin{array}{l}k-1 \\ j-1\end{array}\right) \frac{1}{n^k},\quad  x\to +\infty.
$$
\end{Lemma}
\begin{proof}  
By Binomial Theorem, we expand $f(n+1)$ in a power series with respect to $n^{-1}$, i.e.,
 $$
\sum_{i=1}^{\infty} \frac{a_i}{(n+1)^i}=\sum_{i=2}^{\infty} \frac{a_i}{n^i}\left(1+\frac{1}{n}\right)^{-i}=\sum_{i=1}^{\infty} \frac{a_i}{n^i} \sum_{j=0}^{\infty}\left(\begin{array}{c}
-i \\
j
\end{array}\right) \frac{1}{n^j}.
$$
Since
\begin{eqnarray*}
 \left(\begin{array}{l}
-i \\
j
\end{array}\right) &=&\frac{(-i)(-i-1) \cdots(-i-(j-1))}{j !} \\
& =& (-1)^j \frac{(i)(i+1) \cdots(i+j-1)}{j !} \\
& = &(-1)^j \frac{(i+j-1) !}{j !(j-1) !}\\
&=&(-1)^j\left(\begin{array}{l}
i+j-1 \\
j
\end{array}\right).
\end{eqnarray*}
Then
\begin{eqnarray*}
\sum_{i=1}^{\infty} \frac{a_i}{n^i} \sum_{j=0}^{\infty}\left(\begin{array}{l}-i \\ j\end{array}\right) \frac{1}{n^j}
& =& \sum_{i=1}^{\infty} \frac{a_i}{n^i} \sum_{j=0}^{\infty}(-1)^j\left(\begin{array}{l}i+j-1 \\ j\end{array}\right) \frac{1}{n^j}\\
& =& \sum_{k=1}^{\infty} \sum_{j=1}^{k} a_j(-1)^{k-j}\left(\begin{array}{l}k-1 \\ k-j\end{array}\right) \frac{1}{n^k}\\
& =& \sum_{k=1}^{\infty} \sum_{j=1}^{k} a_j(-1)^{k-j}\left(\begin{array}{l}k-1 \\ j-1\end{array}\right) \frac{1}{n^k}.
\end{eqnarray*}
Consequently,
$$
\sum_{k=1}^{\infty} \frac{a_k}{(n+1)^k}=\sum_{k=1}^{\infty} \sum_{j=1}^{k} a_j(-1)^{k-j}\left(\begin{array}{l}k-1 \\ j-1\end{array}\right) \frac{1}{n^k}.
$$

\end{proof}

With the similar arguments, we have the following result.

\begin{Lemma}\label{BackShift}
Suppose that $f(n-1)$ has  asymptotic power series with respect to $(n-1)^{-k}$
$$
f(n-1) \sim \sum_{k=1}^{\infty} \frac{a_k}{(n-1)^k},\quad  x\to +\infty.
$$
Then $f(n-1)$ has a power expansion  with respect to $n^{-k}$ as follows
$$
f(n-1) \sim \sum_{k=1}^{\infty} \frac{a_k}{(n-1)^k}=\sum_{k=1}^{\infty} \sum_{j=1}^{k} a_j\left(\begin{array}{l}
k-1 \\
j-1
\end{array}\right) \frac{1}{n^k},\quad  x\to +\infty.
$$
\end{Lemma}

\section{Asymptotic expansions of sequences}

In this section, we will give a general method  to obtain asymptotic expansion formulas for three kinds of sequences.

\begin{Theorem}\label{ratioAPP2}
Suppose that $\left\{x_n\right\}_{n \geq  1}$ and $\left\{y_n\right\}_{n \geq  1}$ be two sequences such that
$
\lim _{n \rightarrow \infty}  x_n/y_n=1,
$
and
$$
 \frac{x_n}{y_n}-\frac{x_{n+1}}{y_{n+1}}=\sum_{j=0}^{\infty} \frac{a_j}{n^j},\quad a_0=a_1=0,\quad 
x_n = y_n \left(\sum_{j=0}^{\infty} \frac{b_j}{n^j}\right), \quad \text { as } n \rightarrow+\infty,
$$
where the coefficients $b_j$'s are given by the following recursive formula:
\begin{equation}\label{r1}
b_0=1,b_k=\frac{1}{k}\left(a_{k+1}+\sum_{j=1}^{k-1}(-1)^{j+k}\left(\begin{array}{c}
k \\
j-1
\end{array}\right)  b_{j}\right), \quad   k=1,2,3,\ldots.  
\end{equation}
\end{Theorem}

\Proof
Let 
$$
P_m{(n)}
=\sum_{j=0}^m\frac{b_j}{n^j},\quad \frac{x_n}{y_n}=P_m{(n)}+o\left(\frac{1}{n^m}\right), m\geq 1.
$$
Since $\lim _{n \rightarrow \infty} x_n/y_n=1$, it implies that  
$b_0=1$. Let 
$$
P_m{(n+1)}
=\sum_{j=0}^m\frac{c_j}{n^j}+o\left(\frac{1}{n^m}\right).
$$
By Lemma \ref{ForwardShift}, we have
\begin{eqnarray*}
P_m{(n+1)}&=&1+\sum_{k=1}^{m} \sum_{j=1}^{k} (-1)^{k-j}\left(\begin{array}{l}k-1 \\ j-1\end{array}\right) \frac{b_j}{n^k}  +o\left(\frac{b_k}{n^m}\right).
\end{eqnarray*}
Then
\begin{eqnarray*}
\frac{ x_n  }{y_n }-\frac{  x_{n+1}}{ y_{n+1}}&=&
\left(P_m(n)+o\left(\frac{1}{n^m}\right)\right)-\left(P_m(n+1)+o\left(\frac{1}{(n+1)^m}\right)\right)\\
&=& P_m(n)- P_m(n+1)+o\left(\frac{1}{n^m}\right)\\
&=&  \sum_{k=0}^m\frac{a_k}{n^k}+o\left(\frac{1}{n^m}\right).
 \end{eqnarray*}
Thus
$$
a_0=a_1=0,a_k=
(-1)^k\left(\sum_{j=1}^{k-1}(-1)^{j+1}\left(\begin{array}{c}
k-1 \\
j-1
\end{array}\right)  b_{j}\right), \quad 2 \leq k \leq m.
$$
Then we can obtain (\ref{r1}).

\begin{Theorem}\label{ratioAPP3new}
Suppose that $\left\{x_n\right\}_{n \geq  1}$ and $\left\{y_n\right\}_{n \geq  1}$ be two sequences such that $x_n\sim y_n$, and
$$
\frac{ x_n  }{y_n }\cdot\frac{  x_{n-1}}{ y_{n-1}}=\sum_{j=0}^{\infty} \frac{a_j}{n^j},
 \quad a_0=1,\quad 
 x_n =y_n \left(\sum_{j=0}^{\infty} \frac{b_j}{n^j}\right), \quad \text { as } n \rightarrow+\infty.
$$
Then the coefficients $b_j$'s are given by the following recursive formula:
\begin{equation}\label{r21}
c_0=1,  \quad 
b_1=\frac{a_1}{2},  \quad c_i=\sum_{j=1}^i\left(\begin{array}{c}
i-1 \\
j-1
\end{array}\right)  b_{j},
\quad i=1,2,...,k-1.
\end{equation}
\begin{equation}\label{r22}
b_0=1,  \quad 
b_k=\frac{1}{2}\left(a_k-\sum_{i=1}^{k-1} b_{k-i}c_i
-\sum_{j=1}^{k-1} \left(\begin{array}{c}
k-1 \\
j-1
\end{array}\right)  b_{j}
\right),
\quad k\geq 2.
\end{equation}

\end{Theorem}

\Proof
Let 
$$
P_m{(n)}
=\sum_{j=0}^m\frac{b_j}{n^j},\quad \frac{x_n}{y_n}=P_m{(n)}+o\left(\frac{1}{n^m}\right), m\geq 1.
$$
Since $\lim _{n \rightarrow \infty} x_n/y_n=1$, it implies that  
$b_0=1$. Let 
$$
P_m{(n-1)}
=\sum_{j=0}^m\frac{c_j}{n^j}+o\left(\frac{1}{n^m}\right).
$$
By Lemma \ref{BackShift}, we have
\begin{eqnarray*}
P_m{(n-1)}
&=&1+\sum_{k=1}^m\left(\sum_{j=1}^k\left(\begin{array}{c}
k-1 \\
j-1
\end{array}\right)  b_{j}\right)\frac{1}{n^k}+o\left(\frac{1}{n^m}\right).    
\end{eqnarray*}
Consequently,
$$
c_0=1,  \quad c_k=\sum_{j=1}^k\left(\begin{array}{c}
k-1 \\
j-1
\end{array}\right)  b_{j},
\quad k\geq 1.
$$
Then
\begin{eqnarray*}
\frac{ x_n  }{y_n }\cdot\frac{  x_{n-1}}{ y_{n-1}}&=&
\left(P_m(n)+o\left(\frac{1}{n^m}\right)\right)\left(P_m(n-1)+o\left(\frac{1}{(n-1)^m}\right)\right)\\
&=&  \sum_{k=0}^m\left(\sum_{j=0}^k b_j c_{k-j}\right) \frac{1}{n^k}+o\left(\frac{1}{n^m}\right)\\
&=&  \sum_{k=0}^m\frac{a_k}{n^k}+o\left(\frac{1}{n^m}\right).
 \end{eqnarray*}
Thus
$$
b_0=a_0=1,\quad \sum_{j=0}^{k} b_j c_{k-j}=a_k,
 k=1,2,\ldots.
$$
Together with
$$
c_0=1,  \quad c_k=\sum_{j=1}^k\left(\begin{array}{c}
k-1 \\
j-1
\end{array}\right)  b_{j},
\quad k\geq 1,
$$
we can obtain (\ref{r21}) and (\ref{r22}).

\begin{Theorem}\label{ratioAPP4}
 Let $\left\{x_n\right\}_{n \geq  1}$ and $\left\{y_n\right\}_{n \geq  1}$ be two sequences such that
$x_n\sim y_n$. 
Suppose that
$$
\frac{ x_n  }{y_n }\bigg/\frac{  x_{n-1}}{ y_{n-1}}=\sum_{j=0}^{\infty} \frac{a_j}{n^j}, \quad a_0=1,
 \quad a_1=0,\quad 
 x_n =y_n \left(\sum_{j=0}^{\infty} \frac{b_j}{n^j}\right), \quad \text { as } n \rightarrow+\infty.
$$
Then the coefficients $b_j$'s are given by the following recursive formula:
\begin{equation}\label{r31}
c_0=1, \quad b_0=1, \quad b_1=-a_2, \quad c_i=\sum_{j=1}^i\left(\begin{array}{c}
i-1 \\
j-1
\end{array}\right)  b_{j},
\quad i=1,2,...,k-1,
\end{equation}
and
\begin{equation}\label{r32}
b_k=-\frac{1}{ k}\left[\sum_{j=0}^{k-1} c_j a_{k+1-j}+  \sum_{j=1}^{k-1}\left(\begin{array}{c}
k \\
j-1
\end{array}\right)  b_j\right].
\end{equation}
\end{Theorem}

\Proof
Let 
$$
P_m{(n)}
=\sum_{j=0}^m\frac{b_j}{n^j},\quad \frac{x_n}{y_n}=P_m{(n)}+o\left(\frac{1}{n^m}\right), m\geq 1.
$$
Since $\lim _{n \rightarrow \infty} x_n/y_n=1$, it implies that  
$b_0=1$. Let 
$$
P_m{(n-1)}
=\sum_{j=0}^m\frac{c_j}{n^j}+o\left(\frac{1}{n^m}\right).
$$
Consequently,
$$
c_0=1, \quad c_k=\sum_{j=1}^k\left(\begin{array}{c}
k-1 \\
j-1
\end{array}\right)  b_{j},\quad k\geq 1.
$$
Then
\begin{eqnarray*}
\frac{ x_n  }{y_n }&=&\frac{  x_{n-1}}{ y_{n-1}}\left(\sum_{j=0}^{\infty} \frac{a_j}{n^j}\right)\\
&=&
\left(P_m(n-1)+o\left(\frac{1}{(n-1)^m}\right)\right)\left(\sum_{j=0}^{\infty} \frac{a_j}{n^j}\right)\\
&=&  \sum_{k=0}^m\left(\sum_{j=0}^k c_j a_{k-j}\right) \frac{1}{n^k}+o\left(\frac{1}{n^m}\right)\\
&=&  \sum_{k=0}^m\frac{b_k}{n^k}+o\left(\frac{1}{n^m}\right).
 \end{eqnarray*}
Therefore
$$
b_0=a_0=1,\quad \sum_{j=0}^{k} c_j a_{k-j}=b_k,
\quad 
 k=1,2,\ldots.
$$
Together with
$$
c_0=1,  \quad c_k=\sum_{j=1}^k\left(\begin{array}{c}
k-1 \\
j-1
\end{array}\right)  b_{j},
\quad k\geq 1,
$$
we can obtain 
(\ref{r31}) and (\ref{r32}).

\section{Examples and applications}

In this section, we present some examples relating to the construction of the asymptotic expansion  formulas for some well-known sequences.

\begin{Example}\label{EulerConstant}
Let  $\gamma$ be the Euler’s constant. 
Then the sequence 
$$
x_n=\sum_{k=1}^n \frac{1}{k}-\ln n
$$
has the following asymptotic expansion (cf. \cite{CC2014,Knuth1962})
\begin{eqnarray*}
 x_n &=&\sum_{k=1}^n \frac{1}{k}-\ln n =
  \sum_{j=0}^{\infty} \frac{b_j}{n^j}
 \\
  &= & \gamma+\frac{1}{2 n}-\frac{1}{12 n^2}+\frac{1}{120 n^4}-\frac{1}{252 n^6}+\frac{1}{240 n^8}-\frac{1}{132 n^{10}}+o\left(\frac{1}{n^{10}}\right)
,  
\end{eqnarray*}
where  the  coefficients $b_j$'s can be calculated recursively by the formula:
\begin{equation}\label{EluerCoef}
  b_0=\gamma,\quad b_1=1/2,\quad
  b_k=\frac{(-1)^{k+1}}{k+1}+\frac{(-1)^{k}}{k}
\sum_{j=1}^{k-1}(-1)^{j+1}\left(\begin{array}{c}
k \\
j-1
\end{array}\right)  b_{j}, \quad   k\geq 2.  
\end{equation}

\end{Example}

{\bf Solution.}
  Since
$$
\lim _{n \rightarrow \infty} x_n=\lim _{n \rightarrow \infty}\sum_{k=1}^n \frac{1}{k}-\ln n=\gamma.
$$
Take $y_n=\gamma$. Then 
\begin{eqnarray*}
    \frac{x_n}{y_n}-\frac{x_{n+1}}{y_{n+1}}
    &=&\frac{1}{\gamma}\left(-\frac{1}{n+1}-\ln n+\ln (n+1)\right)
   =\frac{1}{\gamma}\left(\ln(1+\frac{1}{n})-\frac{1}{n}\cdot \frac{1}{1+\frac{1}{n}}\right)
     \\
          &=&\frac{1}{\gamma}\left(\sum_{k=1}^{\infty} \frac{(-1)^{k-1}}{k n^k}-\sum_{k=1}^{\infty} \frac{(-1)^{k-1}}{n^k}\right)= \frac{1}{\gamma}\sum_{k=0}^{\infty} \frac{a_k}{n^k},
 \end{eqnarray*}
where 
$$
a_0=a_1=0,a_k=\frac{(-1)^k(k-1)}{k},\quad k=2,3,\ldots.
$$
Then the equations (\ref{EluerCoef}) follows directly from  Theorem 3.1.
On the other hand, the Euler-Maclaurin summation (e.g. \cite{Knuth1962}) can be used to have a complete asymptotic expansion of the harmonic numbers. We have 
$$
H_n-\log (n) =\gamma+\frac{1}{2 n}-\sum_{k \geq 1} \frac{B_{2 k}}{2 k} \frac{1}{n^{2 k}}
$$
where the $B_{2 k}$ are the Bernoulli numbers. 
Thus  these results leads to the   identity with respect to Bernoulli number $B_j=-jb_j$ for $j=2,3,\ldots$.

\begin{Example}\label{DirichletIntegral1}
Let  $\{a_j\}_{j\geq 0}$ be given by the formula
$$
a_0=1,a_1=-\frac{1}{2},a_j=(-1)^j \frac{\left(\frac{1}{2}\right)_j }{j!},\quad j=2,3,\ldots.
$$
Consider 
the integral
$$
I_n =\frac{1}{\pi} \int_{-\pi / 2}^{\pi / 2}(\cos t)^n \mathrm{~d} t.
$$
Then the integral $I_n$ has the following asymptotic expansion (cf. \cite[pp. 39-40]{Young2017}, \cite[p. 44]{LBC2010})
\begin{eqnarray*}
  I_n &=&
 \sqrt{\frac{2}{\pi n}}
\left(\sum_{j=0}^{\infty} \frac{b_j}{n^j}\right)
 \\
  &=&
\sqrt{\frac{2}{\pi n}}
\left(1-\frac{1}{4n}+\frac{1}{32n^2}+\frac{5}{128n^3}-\frac{21}{2048n^4}-\frac{399}{8192n^5}+o\left(\frac{1}{n^5}\right)\right)
,  
\end{eqnarray*}
where the coefficients $b_j$'s can be calculated recursively by the formula:
\begin{equation}
\label{triintn}
\left\{\begin{array}{l}
c_0=1,  \quad 
b_1=\frac{a_1}{2},  \quad c_i=\sum\limits_{j=1}^i\left(\begin{array}{c}
i-1 \\
j-1
\end{array}\right)  b_{j},
\quad i=1,2,...,k-1. \\
b_0=1,  \quad 
b_k=\frac{1}{2}\left(a_k-\sum\limits_{i=1}^{k-1} b_{k-i}c_i
-\sum\limits_{j=1}^{k-1} \left(\begin{array}{c}
k-1 \\
j-1
\end{array}\right)  b_{j}
\right),
\quad k\geq 2.
\end{array}\right.
\end{equation}
\end{Example}

{\bf Solution.}
Note that $I_n$ is strictly decreasing. And it can be proved by induction
$$
n I_n I_{n-1}=\frac{2}{\pi}, \quad n \geq  1.
$$
Take $x_n=I_n$. Since
$x_n\sim 
 \sqrt{\frac{2}{\pi n}} $ as  $n \rightarrow+\infty$, one can take $y_n= \sqrt{\frac{2}{\pi n}}$. Then
 $$
 \frac{ x_n  x_{n-1}}{y_n y_{n-1}}=\frac{ \frac{2}{\pi n}}{ \sqrt{\frac{2}{\pi n}}  \sqrt{\frac{2}{\pi (n-1)}}}
 =\left(1-\frac{1}{n}\right)^{\frac{1}{2}}
 = \sum_{j=0}^{\infty} \frac{a_j}{n^j},
  $$
where 
$$
a_0=1,a_j=(-1)^j \frac{\left(\frac{1}{2}\right)_j}{j!},\quad j=2,3,\ldots.
$$
Then the equations (\ref{triintn}) follows directly from  Theorem 3.2.

 We report, in Table 1, some numerical values obtained with a rounding off at the tenth decimal place on
the parameter $n=11$.

\begin{table}[ht]
\begin{center}

\caption{Asymptotic estimate $I_{n,k}:=\sqrt{\frac{2}{\pi n}}
\left(\sum\limits_{j=0}^{k} \frac{b_j}{n^j}\right)$}
\begin{tabular}{llllll}
\hline $I_{11}$ & $ I_{11,1}$ & $I_{11,2}$ & $I_{11,3}$ & $I_{11,4}$  & $I_{11,5}$ \\
\hline $0.235172672$ & $0.235103718$ & $0.235165849$ & $0.23517291$ & $0.235172741$  & $0.235172669$\\
\hline
\end{tabular}
\end{center}
\end{table}

\begin{Lemma}\label{seqapp} (cf. \cite{ChenEV2013,ElezovicVuk2014})
Let $ n\in \mathbb{N} \backslash\{0\}$. Then
\begin{eqnarray*}
\left(1-\frac{1}{n^2}\right)^n
   & = & \sum_{k=0}^{\infty}  \frac{s_k }{n^k},
\end{eqnarray*}
where $s_0=1$ and
$$
s_k=-\frac{1}{k} \sum_{j=1}^{\lfloor \frac{k+1}{2}\rfloor} \frac{2j-1}{j} s_{k+1-2j}, \quad k =1,2,\cdots.
$$
\end{Lemma}
\proof
Let $n=x$. Then 
\begin{eqnarray*}
\left(1-\frac{1}{n^2}\right)^n
&=& 
\exp\left(n\ln{\left(1- \frac{1}{n^{2}}\right)}\right)
\\&=& 
\exp\left(\sum_{k=1}^{\infty} \frac{-1}{k } x^{1-2k}\right)
\\&=& 
\sum_{k=0}^{\infty} s_k x^{-k}
\end{eqnarray*}
Let $x$ tend to infinity. Then one has $s_0=1$.
Differentiating the last equation above with respect to $x$, we get
\begin{eqnarray*}
\sum_{k=1}^{\infty} -k s_k x^{-k-1} & = &\left(\sum_{k=1}^{\infty} \frac{2k-1}{k } x^{-2k}\right)\left(\sum_{k=0}^{\infty} s_k x^{-k}\right) \\
& =& \sum_{k=1}^{\infty}\left(\sum_{j=1}^{\lfloor \frac{k+1}{2}\rfloor} \frac{2j-1}{j} s_{k+1-2j}\right) x^{-k-1}
\end{eqnarray*}
where $s_0=1$ and
$$
s_k=-\frac{1}{k} \sum_{j=1}^{\lfloor \frac{k+1}{2}\rfloor} \frac{2j-1}{j} s_{k+1-2j}, \quad k =1,2,\cdots.
$$

\begin{Example}\label{expn}
It is well-known that the classic limit
$$
e= \lim _{n \rightarrow \infty}\left(1+\frac{1}{n}\right)^n
$$
For $k=0,1,2,\ldots$, let  $a_k=\sum_{j=0}^{k}s_j$ where
  $s_0=1$ and
$$
s_j=-\frac{1}{j} \sum_{i=1}^{\lfloor \frac{j+1}{2}\rfloor} \frac{2i-1}{i} s_{j+1-2i}, \quad j =1,2,\cdots.
$$
Then the sequence $x_n=(1+1/n)^n$ has the following asymptotic expansion  (cf. \cite{BK1998,CC2014AMM})
\begin{eqnarray*}
  x_n &=& (1+1/n)^n
=
e
\left(\sum_{j=0}^{\infty} \frac{b_j}{n^j}\right)
\\
&=& 
e\left(1-\frac{1}{2 n}+\frac{11}{24 n^2}-\frac{7}{16 n^3}+\frac{2447}{5760 n^4}-\frac{959}{2304 n^5}+\frac{238043}{580608 n^6}-\cdots\right)
,  
\end{eqnarray*}
where the coefficients $b_j$ can be calculated recursively by the formula:
\begin{equation}
\label{een}
\left\{\begin{array}{l}
c_0=1,\,b_0=1,\,b_1=-a_2, \quad c_i=\sum\limits_{j=1}^i\left(\begin{array}{c}
i-1 \\
j-1
\end{array}\right)  b_{j},
\quad i=1,2,...,k-1,\\
 b_k=-\frac{1}{k}\left[\sum\limits_{j=0}^{k-1} c_j a_{k+1-j}+ \sum\limits_{j=1}^{k-1}\left(\begin{array}{c}
k \\
j-1
\end{array}\right)  b_j\right].
\end{array}\right.
\end{equation}
\end{Example}

\proof
Take $y_n=e$, then it follows from Lemma \ref{seqapp} that
\begin{eqnarray*}
\frac{ x_n  }{y_n }\bigg/\frac{  x_{n-1}}{ y_{n-1}}&=&
\frac{ (1+\frac{1}{n})^n }{ (1+\frac{1}{n-1})^{n-1} }= \frac{ (1+\frac{1}{n})^n }{ (1+\frac{1}{n-1})^{n} }\cdot \left(1+\frac{1}{n-1}\right)
\\&=& 
\left(1-\frac{1}{n^2}\right)^n\left(1+\frac{\frac{1}{n}}{1-\frac{1}{n}}\right)
\\&=& 
\left(\sum_{k=0}^{\infty}\frac{s_k}{n^k}\right)\left(\sum_{k=0}^{\infty}\frac{1}{n^k}\right)
\\
 &=& \sum_{k=0}^{\infty} \frac{a_k}{n^k},
\end{eqnarray*}
where $a_0=1$, $a_1=0$, for $k=2,3,\ldots$,
$a_k=\sum_{j=0}^{k}s_j$ where $s_j$'s are defined by Lemma \ref{seqapp}.
Then the equations (\ref{een}) follows directly from  Theorem 3.3.

\begin{Example}\label{PartIntegral1}
Let  $\{a_j\}_{j\geq 0}$ be given by the formula
$$
a_0=1,a_1=0, a_k=-\frac{1}{2}-\frac{1}{2}\sum_{j=1}^{k-1} \frac{(-1)^j\left(-\frac{1}{2}\right)_j}{j!}+\frac{(-1)^k\left(-\frac{1}{2}\right)_k}{k!},\quad k=2,3,\ldots.
$$
Consider 
the integral
$$
 J_n =\int_{0}^{\infty} \frac{1}{\left(1+t^{2}\right)^{n}} \mathrm{~d} t.
$$
Then the integral $J_n$ has the following asymptotic expansion 
\begin{eqnarray*}
  J_n &=& 
\frac{1}{2} \sqrt{\frac{\pi}{n}}
\left(\sum_{j=0}^{\infty} \frac{b_j}{n^j}\right)
\\
&=& 
\frac{1}{2} \sqrt{\frac{\pi}{n}}
\left(1+\frac{3}{8n}+\frac{25}{128n^2}+\frac{105}{1024n^3}+\frac{302}{5965n^4}+o\left(\frac{1}{n^4}\right)\right)
,  
\end{eqnarray*}
where the coefficients $b_j$ can be calculated recursively 
by the formula:
\begin{equation}
\label{ratintn}
\left\{\begin{array}{l}
c_0=1,\,b_0=1,\,b_1=-a_2, \quad c_i=\sum\limits_{j=1}^i\left(\begin{array}{c}
i-1 \\
j-1
\end{array}\right)  b_{j},
\quad i=1,2,...,k-1,\\
 b_k=-\frac{1}{k}\left[\sum\limits_{j=0}^{k-1} c_j a_{k+1-j}+ \sum\limits_{j=1}^{k-1}\left(\begin{array}{c}
k \\
j-1
\end{array}\right)  b_j\right].
\end{array}\right.
\end{equation}
\end{Example}

\proof
With integration by parts, we get
$$
J_{n}=\frac{2 n-3}{2 n-2} J_{n-1}.
$$
Take $x_n=J_n$. By Laplace's method, one can obtain
$x_n\sim y_n=
\frac{1}{2} \sqrt{\frac{\pi}{n}} $ as  $n \rightarrow+\infty$.  then
\begin{eqnarray*}
\frac{ x_n  }{y_n }\bigg/\frac{  x_{n-1}}{ y_{n-1}}&=&\frac{ (2 n-3) \cdot  \frac{1}{2}\sqrt{\frac{\pi}{n-1}}}{ (2 n-2) \cdot \frac{1}{2} \sqrt{\frac{\pi}{n}}}= \frac{(2 n-3) }{(2 n-2) }\sqrt{\frac{n}{n-1}}
\\&=& 
\frac{1-\frac{3}{2n} }{1-\frac{1}{n}} \left(1-\frac{1}{n}  \right)^{-\frac{1}{2}}
\\&=& 
\left(1-\frac{1}{2}\sum_{k=1}^{\infty}\frac{1}{n^k}\right)\left(1+\sum_{j=1}^{\infty}\frac{\left(-\frac{1}{2}\right)_j}{j!}\frac{(-1)^j}{n^j}\right)
\\
 &=& \sum_{k=0}^{\infty} \frac{a_k}{n^k},
\end{eqnarray*}
where $a_0=1,a_1=0$ and
$$
 a_k=-\frac{1}{2}-\frac{1}{2}\sum_{j=1}^{k-1} \frac{(-1)^j\left(-\frac{1}{2}\right)_j}{j!}+\frac{(-1)^k\left(-\frac{1}{2}\right)_k}{k!},\quad k=2,3,\ldots.
$$
Then the equations (\ref{ratintn}) follows directly from  Theorem 3.3.

In Table 2, some numerical values obtained with a rounding off at the tenth decimal place on
the parameter $n=10$.

\begin{table}[ht]
\begin{center}
\caption{Asymptotic estimate $J_{n,k}:=\frac{1}{2}\sqrt{\frac{\pi}{n}}
\left(\sum\limits_{j=0}^{k} \frac{b_j}{n^j}\right)$}
\begin{tabular}{llllll}
\hline $J_{10}$ & $ J_{10,0}$ & $J_{10,1}$ & $J_{10,2}$ & $J_{10,3}$  & $J_{10,4}$ \\
\hline $0.291336507$ & $0.280249561$ & $0.290758919$ & $0.291306282$ & $0.291335018$  & $0.291336437$\\
\hline
\end{tabular}
\end{center}
\end{table} 

\vskip 2mm

\parskip=0.4cm
\vspace{1cm}

\end{document}